\documentclass[10pt, reqno]{amsart}

\usepackage{amsfonts}
\usepackage{amsthm}
\usepackage{amsmath}
\usepackage{oldgerm}

\newtheorem{thm}{Theorem}[section]
\newtheorem{lemma}[thm]{Lemma}
\newtheorem{cor}[thm]{Corollary}
\newtheorem{fact}[thm]{Theorem}

\theoremstyle{definition}
\newtheorem{definition}[thm]{Definition}
\newtheorem{ex}[thm]{Example}

\theoremstyle{remark}
\newtheorem{remark}[thm]{Remark}

\numberwithin{equation}{section}

\DeclareMathOperator{\rank}{rank}
\DeclareMathOperator{\Range}{Range}
\DeclareMathOperator{\ess}{ess}

\DeclareMathOperator{\dist}{dist}
\DeclareMathOperator{\ind}{ind}
\DeclareMathOperator{\mySpan}{span}
\DeclareMathOperator{\codim}{codim}

\newcommand{\Hk}{H_{(k)}^{\infty}}
\newcommand{\Ek}{E^{(k)}}
\newcommand{\fS}{\textfrak{S}}

\newcommand{\cB}{\mathcal{B}}
\newcommand{\cE}{\mathcal{E}}

\newcommand{\cH}{\mathcal{H}}
\newcommand{\cK}{\mathcal{K}}
\newcommand{\cL}{\mathcal{L}}

\newcommand{\cM}{\mathcal{M}}

\newcommand{\C}{\mathbb{C}}
\newcommand{\M}{\mathbb{M}}

\newcommand{\T}{\mathbb{T}}
\newcommand{\D}{\mathbb{D}}
\newcommand{\pP}{\mathbb{P}}

\newcommand{\m}{\boldsymbol{m}}
\newcommand{\dm}{d\m}

\newcommand{\Myspan}{{\rm span}}
\newcommand{\Zero}{\mathbb{O}}
\newcommand{\Mydef}{\stackrel{\mbox{\footnotesize{\rm def}}}{=}}

\begin{document}
\title[An index formula in connection with meromorphic approximation]{An index formula in connection with 
meromorphic approximation}

\author{Alberto A. Condori}
\address{Department of Chemistry and Mathematics\\ 
Florida Gulf Coast University\\
10501 FGCU Boulevard South\\
Fort Myers, FL 33965}
\email{acondori@fgcu.edu}



\begin{abstract}
			Let $\Phi$ be a continuous $n\times n$ matrix-valued function on the unit circle $\T$ such that 
			the $(k-1)$th singular value of the Hankel operator with symbol $\Phi$ is greater than the $k$th 
			singular value.  In this case, it is well-known that $\Phi$ has a unique superoptimal meromorphic 
			approximant $Q$ in $H^{\infty}_{(k)}$; that is, $Q$ has at most $k$ poles in the unit disc $\mathbb{D}$ 
			(in the sense that the McMillan degree of $Q$ in $\mathbb{D}$ is at most $k$) and $Q$ minimizes the 
			essential suprema of singular values $s_{j}\left((\Phi-Q)(\zeta)\right)$, $j\geq0$,	with respect to 
			the lexicographic ordering.  For each $j\geq 0$, the essential supremum of 
			$s_{j}\left((\Phi-Q)(\zeta)\right)$ is called the $j$th superoptimal singular value of degree $k$ 
			of $\Phi$.  We prove that if $\Phi$ has $n$ non-zero superoptimal singular values of degree $k$, 
			then the Toeplitz operator $T_{\Phi-Q}$ with symbol $\Phi-Q$ is Fredholm and has index
			\[ \ind T_{\Phi-Q}=\dim\ker T_{\Phi-Q}=2k+\dim\mathcal{E}, \]
			where $\mathcal{E}=\{ \xi\in\ker H_{Q}: \|H_{\Phi}\xi\|_{2}=\|(\Phi-Q)\xi\|_{2}\}$ and $H_{\Phi}$ 
			denotes the Hankel operator with symbol $\Phi$.  This result can in fact be extended from continuous 
			matrix-valued functions to the wider class of $k$-\emph{admissible} matrix-valued functions, i.e. 
			essentially bounded $n\times n$ matrix-valued functions $\Phi$ on $\T$ for which the essential norm 
			of the Hankel operator $H_{\Phi}$ is strictly less than the smallest non-zero superoptimal singular 
			value of degree $k$ of $\Phi$.
\end{abstract}
 
\maketitle

\section{Introduction}

Let $\varphi$ be a bounded measurable function defined on the unit circle $\T$.  For $k\geq0$, 
let $\Hk$ denote the collection of meromorphic functions in the unit disc $\D$ which are bounded near $\T$ 
and have at most $k$ poles in $\D$ (counting multiplicities).  The \emph{Nehari-Takagi problem} is to find 
a $q\in\Hk$ which is closest to $\varphi$ with respect to the $L^{\infty}$-norm, i.e. to find $q\in \Hk$ 
such that
\[	\|\varphi-q\|_{\infty}=\dist_{L^{\infty}}(\varphi, \Hk)=\inf_{f\in\Hk}\|\varphi-f\|_{\infty}.	\]
Any such function $q$ is called a \emph{best approximant in $\Hk$ to $\varphi$}.  

Although a best approximant in $\Hk$ need not be unique in general, if $\varphi$ is a continuous function 
on $\T$,  then uniqueness holds.  Moreover, under this assumption, it can be shown that the function 
defined by $\varphi-q$ has constant modulus (equal to $s_{k}(H_{\varphi})$) a.e. on $\T$, the Toeplitz 
operator $T_{\varphi-q}$ is Fredholm, and 
\begin{equation}\label{scalarIndex}
			\ind T_{\varphi-q}=2k+\mu,
\end{equation}
where $\mu$ denotes the multiplicity of the singular value $s_{k}(H_{\varphi})$ of the Hankel operator 
$H_{\varphi}$ with symbol $\varphi$ (e.g. see Chapter 4 in \cite{Pe1}).  In fact, the best meromorphic 
approximant to $\varphi$ in $\Hk$ is the unique function in $\Hk$ that has these three properties.

The index formula in $(\ref{scalarIndex})$ not only provides a uniqueness criterion for the best 
meromorphic approximant in $\Hk$, it also appears in applications such as the study of singular 
values of Hankel operators with perturbed symbols (see \cite{Pe2} or Chapter 7 in \cite{Pe1}).
This index formula can also be used to obtain a sharp estimate on the degree of the best meromorphic 
approximant to a given (scalar) rational function \cite{Pe1}.  

Note that in the case of $2\times 2$ matrix-valued rational functions, sharp estimates have also been 
obtained but only for their \emph{analytic} approximants \cite{PV}.  These estimates were made without
use of the index formula argument used for scalar functions.

Thus, the main focus of this paper is to obtain an analogous index formula for matrix-valued functions 
on $\T$.  Unlike the scalar-case, if $\Phi$ is a continuous matrix-valued function on $\T$, then $\Phi$ 
may not have a unique best meromorphic approximant $Q$ with at most $k$ poles in $\D$ (i.e. the McMillan 
degree of $Q$ in $\D$ is at most $k$).  However, the uniqueness of a \emph{superoptimal} meromorphic 
approximant $Q$ having at most $k$ poles in $\D$ does hold under the additional assumption that
$s_{k}(H_{\Phi})<s_{k-1}(H_{\Phi})$.  Therefore, the natural question arises whether the index formula 
in $(\ref{scalarIndex})$ holds for such matrix-valued functions $\Phi$ with superoptimal meromorphic 
approximant $Q$; that is, does
\begin{equation}\label{indexMatrix}
			\ind T_{\Phi-Q}=2k+\mu
\end{equation}
hold?

Recall that for continuous $n\times n$ matrix-valued functions $\Psi$, the Toeplitz operator $T_{\Psi}$ is 
Fredholm if and only if $\det\Psi$ does not vanish a.e. on $\T$.  It follows from known results regarding
the error term $\Phi-Q$ that a necessary and sufficient condition for the Toeplitz operator $T_{\Phi-Q}$ 
to be Fredholm (when $\Phi$ is continuous) is that \emph{all superoptimal singular values of degree $k$ 
of $\Phi$ are non-zero}.  Unfortunately, as shown in Example \ref{theEx} below, the index formula in 
$(\ref{indexMatrix})$ fails to hold under this additional assumption.

\begin{ex}\label{theEx}
			Consider the matrix-valued function
			\begin{equation*}
						\Phi=\frac{1}{\sqrt{2}}\left(	\begin{array}{cc}
																								\bar{z}^{5}+\frac{1}{3}\bar{z}	&	-\frac{1}{3}\bar{z}^{2}\\
																								\bar{z}^{4}											&	\frac{1}{3}\bar{z}
																					\end{array}\right).
			\end{equation*}
			It is not difficult to verify that the non-zero singular values of $H_{\Phi}$ are
			\begin{align*}
						s_{0}(H_{\Phi})=\frac{\sqrt{10}}{3},\,  s_{1}(H_{\Phi})=s_{2}(H_{\Phi})=s_{3}(H_{\Phi})=1,\,
						 s_{4}(H_{\Phi})=\frac{1}{\sqrt{2}},\, \text{ and }\, s_{5}(H_{\Phi})=\frac{1}{3}.
			\end{align*}
			In particular, if $\mu$ denotes the multiplicity of the singular value $s_{1}(H_{\Phi})=1$ of the 
			Hankel operator $H_{\Phi}$, then $2k+\mu=5$.
			
			Using an algorithm due to Peller and Young (\cite{PY2} or section 17 of Chapter 14 in \cite{Pe1}), 
			it can be shown that the superoptimal approximant in $H^{\infty}_{(1)}(\M_{2})$ to $\Phi$ is 
			\begin{equation*}
						Q=\frac{1}{\sqrt{2}}\left(	\begin{array}{cc}
																							\frac{1}{3}\bar{z}	&	\Zero\\
																							\Zero								&	\Zero
																				\end{array}\right).
			\end{equation*}
			However, $\ind T_{\Phi-Q}=\dim\ker T_{\Phi-Q}=6$ (by Theorem 7.4 of Chapter 14 in \cite{Pe1} or 
			Theorem 2.2 in \cite{PY3}), because $\Phi-Q$ admits a (``thematic'') factorization of the form
			\begin{align*}
						\Phi-Q=	\frac{1}{\sqrt{2}}
										\left(	\begin{array}{cc}
																	\bar{z}	&	-1\\
																	1				&	z
														\end{array}\right)
										\left(	\begin{array}{cc}
																	\bar{z}^{4}	&	\Zero\\
																	\Zero				&	\frac{1}{3}\bar{z}^2
														\end{array}\right).
			\end{align*}
			Hence the index formula in $(\ref{indexMatrix})$ fails to hold for this choice of $\Phi$.
\end{ex}

In this paper, we establish the correct analog to the index formula in $(\ref{scalarIndex})$, namely
\[	\ind T_{\Phi-Q}=\dim\ker T_{\Phi-Q}=2k+\dim\cE	\]
for continuous $n\times n$ matrix-valued functions $\Phi$ such that $s_{k}(H_{\Phi})<s_{k-1}(H_{\Phi})$ 
and whose superoptimal singular values of degree $k$ are all non-zero, where
\[	\cE=\{ \xi\in\ker H_{Q}: \|H_{\Phi}\xi\|_{2}=\|(\Phi-Q)\xi\|_{2}\}.	\]
In fact, we prove that our analog holds in the more general case of ``$k$-admissible'' bounded 
matrix-valued functions.  This is accomplished using a result involving ($k$-admissible) matrix-valued 
weights for Hankel operators, the proof of which was inspired by Treil's approach to superoptimal 
approximation (actually the main ideas go back to \cite{T2}).  We also show that if all superoptimal 
singular values of degree $k$ of $\Phi$ are equal, then our index formula agrees with the formula in 
$(\ref{indexMatrix})$ (see Corollary \ref{indexCor}).  This result is obtained using a characterization 
of the space of Schmidt vectors $\Ek(\Phi)$ that correspond to the singular value $s_{k}(H_{\Phi})$ 
of the Hankel operator $H_{\Phi}$ with symbol $\Phi$.  Note that this characterization involves 
\emph{any} best approximant in $\Hk(\M_{n})$ to $\Phi$.  

The organization of the paper is as follows.  All necessary background on superoptimal approximation 
appears in section \ref{BackgroundSection}.  The characterization of the space of Schmidt vectors is 
given in section \ref{schmidtSection}.  We prove in section \ref{FredholmSection} that the Toeplitz 
operator induced by the error term $\Phi-Q$ is Fredholm and establish in section \ref{weightSection}
a result concerning matrix-valued weights for Hankel operators.  Section \ref{indexSection} contains 
proof that our analog to the index formula holds for $k$-admissible matrix-valued functions.

\subsection{Notation and terminology.}  Throughout the paper, we use the following notation and 
terminology:

$\m$ denotes normalized Lebesgue measure on the unit circle $\T$ so that $\m(\T)=1$;

$\Zero$ denotes the matrix-valued function which equals the zero matrix on $\T$ (its size will be clear 
in the context);

$\M_{m,n}$ denotes the space of $m\times n$ matrices equipped with the operator norm $\|\cdot\|_{\M_{m,n}}$ 
and $\M_{n}\Mydef\M_{n,n}$;

$A^{t}$ denotes the transpose of a matrix $A\in\M_{m,n}$;

$X(\M_{m,n})$ denotes the space of $m\times n$ matrix-valued functions on $\T$ whose entries belong to 
a space $X$ of scalar functions on $\T$ and $X(\C^{n})\Mydef X(\M_{n,1})$;

$\displaystyle{\|\Psi\|_{L^{\infty}(\M_{m,n})}\Mydef\ess\sup_{\zeta\in\T}\|\Psi(\zeta)\|_{\M_{m,n}}}$ 
for $\Psi\in L^{\infty}(\M_{m,n})$;

$\Phi^{t}$ denotes the function $\Phi^{t}(\zeta)\Mydef(\Phi(\zeta))^{t}$, $\zeta\in\T$, when $\Phi\in 
L^{\infty}(\M_{m,n})$;

$\cB(X,Y)$ denotes the collection of bounded linear operators $T:X\rightarrow Y$ between normed spaces 
$X$ and $Y$;

if $T\in\cB(X,Y)$, we say that a non-zero vector $x\in X$ is a \emph{maximizing vector of $T$} whenever 
$\|Tx\|_{Y}=\|T\|\cdot\|x\|_{X}$;

$\cH$ and $\cK$ denote Hilbert spaces;

if $T\in\cB(\cH,\cK)$, the \emph{singular values $s_{n}(T)$, $n\geq0$, of $T$} are defined by
\[	s_{n}(T)=\inf\{\|T-R\|: R\in\cB(\cH,\cK), \rank R\leq n \}	\]
and the \emph{essential norm of $T$} is defined by
\[	\|T\|_{\rm e}=\inf\{\|T-K\|: K\in\cB(\cH,\cK), K\text{ is a compact operator }\};	\]

and if $T\in\cB(\cH,\cK)$ and $s$ is a singular value of $T$, a non-zero vector $x\in\cH$ is called a 
\emph{Schmidt vector} corresponding to $s$ whenever $T^{*}T x=s^{2}x$.

\section{Background}\label{BackgroundSection}

\subsection{Best and superoptimal approximation in $\Hk(\M_{m,n})$}\label{bestRemarks}

To introduce the class $\Hk(\M_{m,n})$, we must first define the notion of a finite Blaschke-Potapov 
product.

A matrix-valued function $B\in H^{\infty}(\M_{n})$ is called a finite \emph{Blaschke-Potapov product} 
if it admits a factorization of the form
\begin{equation*}
			B=U B_{1}B_{2}\ldots B_{m},
\end{equation*}
where $U$ is a unitary matrix and, for each $1\leq j\leq m$,
\[	B_{j}=\frac{z-\lambda_{j}}{1-\bar{\lambda}_{j}z}P_{j}+(I-P_{j})	\]
for some $\lambda_{j}\in\D$ and orthogonal projection $P_{j}$ on $\C^{n}$.  The \emph{degree} of the 
Blaschke-Potapov product $B$ is defined to be
\[	\deg B\Mydef \sum_{j=1}^{m}\rank P_{j}.	\]
Alternatively, $B$ is a finite Blaschke-Potapov product of degree $k$ if and only if $B$ admits a 
factorization of the form
\begin{equation}
			B(z)=U_{0}\left(
			\begin{array}{cc}
						\frac{z-a_{1}}{1-\bar{a}_{1}z}	&	\Zero\\
						\Zero														&	I_{n-1}	
			\end{array}\right)U_{1}\ldots U_{k-1}\left(
			\begin{array}{cc}
						\frac{z-a_{k}}{1-\bar{a}_{k}z}	&	\Zero\\
						\Zero														&	I_{n-1}	
			\end{array}\right)U_{k},
\end{equation}
where $a_{1},\ldots,a_{k}\in\D$; $U_{0},U_{1},\ldots,U_{k}$ are constant $n\times n$ unitary matrices; 
and $I_{n-1}$ denotes the $(n-1)\times(n-1)$ identity matrix.

It turns out that every invariant subspace $\cL$ of multiplication by $z$ on $H^{2}(\C^{n})$ of finite 
codimension is of the form $B H^{2}(\C^{n})$ for some Blaschke-Potapov product $B$ of finite degree.  
Moreover, the degree of $B$ equals $\codim \cL$ (e.g. see Lemma 5.1 in Chapter 2 of \cite{Pe1}).

A matrix-valued function $Q\in L^{\infty}(\M_{m,n})$ is said to have \emph{at most $k$ poles in $\D$} 
if there is a finite Blaschke-Potapov product $B$ of degree $k$ such that $QB\in H^{\infty}(\M_{m,n})$.  
We denote the collection of $m\times n$ matrix-valued functions $Q$ that have at most $k$ poles in $\D$ 
by $\Hk(\M_{m,n})$.  

For $Q\in L^{\infty}(\M_{m,n})$ with at most $k$ poles in $\D$, \emph{the McMillan degree of $Q$ in $\D$} 
is the smallest number $j\geq 0$ such that $Q$ has at most $j$ poles in $\D$.  In particular, $\Hk(\M_{m,n})$ 
consists of matrix-valued functions $Q\in L^{\infty}(\M_{m,n})$ which can be written in the form $Q=R+F$ for
some $F\in H^{\infty}(\M_{m,n})$ and some rational $m\times n$ matrix-valued function $R$ with poles in $\D$ 
such that the McMillan degree of $R$ in $\D$ is at most $k$.  In this paper, we do not need (explicitly) the 
general definition of McMillan degree (which omits the restriction to the disc $\D$) and thus refer the 
interested reader to Chapter 2 in \cite{Pe1} for further information regarding McMillan degree.

\begin{definition}
			Let $k\geq 0$.  Given an $m\times n$ matrix-valued function $\Phi\in L^{\infty}(\M_{m,n})$, we say 
			that $Q$ is a \emph{best approximant in $\Hk(\M_{m,n})$ to $\Phi$} if $Q\in\Hk(\M_{m,n})$ and
			\[	\|\Phi-Q\|_{L^{\infty}(\M_{m,n})}=\dist_{L^{\infty}(\M_{m,n})}(\Phi, \Hk(\M_{m,n})).	\]
\end{definition}
Note that by a compactness argument, a matrix-valued function $\Phi\in L^{\infty}(\M_{m,n})$ always has 
a best approximant in $\Hk(\M_{m,n})$.  That is, the set
\[	\Omega_{0}^{(k)}(\Phi)\Mydef\left\{	Q\in\Hk(\M_{m,n}): 
		Q \mbox{ minimizes }\ess\sup_{\zeta\in\T}\|\Phi(\zeta)-Q(\zeta)\|_{\M_{m,n}}	\right\}	\]
is always non-empty (e.g. see section 3 in Chapter 4 of \cite{Pe1}).

As in the case of scalar-valued bounded functions, Hankel operators on Hardy spaces are very useful tools 
in the study of best approximation by matrix-valued functions in $\Hk(\M_{m,n})$.  For a matrix-valued 
function $\Phi\in L^{\infty}(\M_{m,n})$, we define the \emph{Hankel operator $H_{\Phi}$} by 
\[	H_{\Phi}f=\pP_{-}\Phi f,\;\mbox{ for }f\in H^{2}(\C^{n}),	\]
where $\pP_{-}$ denotes the orthogonal projection of $L^{2}(\C^{m})$ onto $H^{2}_{-}(\C^{m})=L^{2}(\C^{m})
\ominus H^{2}(\C^{m})$.  It is well-known (\cite{T1} or section 3 of Chapter 4 in \cite{Pe1}) that
\begin{equation}\label{skTreil}
			\dist_{L^{\infty}(\M_{m,n})}(\Phi, \Hk(\M_{m,n}))=s_{k}(H_{\Phi}).
\end{equation}
However in contrast to the case of scalar-valued functions, it is known that the condition $\|H_{\Phi}\|_{\rm e}
<s_{k}(H_{\Phi})$ does not guarantee uniqueness of a best approximant in $\Hk(\M_{m,n})$ to $\Phi$.

Since the set of best approximants $\Omega_{0}^{(k)}(\Phi)$ to $\Phi$ may contain distinct elements, 
it is natural to refine the notion of optimality if possible to obtain the ``very best'' matrix-valued 
function in $\Omega_{0}^{(k)}(\Phi)$.

\begin{definition}\label{superoptDef}
			Let $k\geq 0$ and $\Phi\in L^{\infty}(\M_{m,n})$.  For $j>0$, define the sets
			\[	\Omega_{j}^{(k)}(\Phi)\Mydef\left\{	Q\in\Omega_{j-1}^{(k)}(\Phi): 
					Q \mbox{ minimizes }\ess\sup_{\zeta\in\T}s_{j}(\Phi(\zeta)-Q(\zeta)) \right\}.	\]
			We say that $Q$ is a \emph{superoptimal approximant in $\Hk(\M_{m,n})$ to $\Phi$} if $Q$ belongs to
			$\displaystyle{\bigcap_{j\geq0}\Omega_{j}^{(k)}(\Phi)=\Omega^{(k)}_{\min\{m,n\}-1}(\Phi)}$ and in 
			this case we define the \emph{superoptimal singular values of degree $k$ of $\Phi$} by 
			\[	t_{j}^{(k)}(\Phi)=\ess\sup_{\zeta\in\T}s_{j}((\Phi-Q)(\zeta))\mbox{ for }j\geq 0.	\]
			In the case $k=0$, we also use the notations $\Omega_{j}(\Phi)$ and $t_{j}(\Phi)$ to denote 
			$\Omega_{j}^{(0)}(\Phi)$ and $t_{j}^{(0)}(\Phi)$, respectively, for $j\geq 0$.
\end{definition}

In \cite{T2}, Treil proved that a \emph{unique} superoptimal approximant $Q$ in $\Hk(\M_{m,n})$ to $\Phi$ 
exists whenever $\Phi\in (H^{\infty}+C)(\M_{m,n})$ and $s_{k}(H_{\Phi})<s_{k-1}(H_{\Phi})$.  (Recall that 
$H^{\infty}+C$ denotes the closed subalgebra of $L^{\infty}$ that consists of functions of the form $f+g$ 
with $f\in H^{\infty}$ and $g\in C(\T)$.)  Shorty after, Peller and Young also proved this result in \cite{PY2} 
using a diagonalization argument which also constructs (in principle) the superoptimal approximant. 

A matrix-valued function $\Phi\in L^{\infty}(\M_{m,n})$ is called \emph{$k$-admissible} if $s_{k}(H_{\Phi})
<s_{k-1}(H_{\Phi})$ and $\|H_{\Phi}\|_{\rm e}$ is strictly less than the smallest non-zero number in the 
set $\{t_{j}^{(k)}(\Phi)\}_{j\geq0}$.  (Note that the statement regarding the singular values of the Hankel 
operator is vacuous when $k=0$.)  For notational simplicity, we refer to $0$-admissible matrix-valued functions 
as \emph{admissible}.  In particular, any matrix-valued function $\Phi$ that belongs to $(H^{\infty}+C)(\M_{m,n})$ 
is admissible because the Hankel operator $H_{\Phi}$ has essential norm equal to zero.

It is now known (see section 17 of Chapter 14 in \cite{Pe1}) that if $\Phi\in L^{\infty}(\M_{m,n})$ is 
$k$-admissible, then $\Phi$ has a unique superoptimal approximant $Q$ in $\Hk(\M_{m,n})$ and
\begin{equation}\label{sNumbersAndT}
			s_{j}((\Phi-Q)(\zeta))=	t_{j}^{(k)}(\Phi)\text{ a.e. }\zeta\in\T, j\geq 0.
\end{equation}

\subsection{Very badly approximable functions}\label{verybadSection}

Let $G\in L^{\infty}(\M_{m,n})$.  We say that $G$ is \emph{very badly approximable} if the matrix-valued 
function $\Zero$ is a superoptimal approximant in $H^{\infty}(\M_{m,n})$ to $G$.

It is well-known that if $G$ is an admissible very badly approximable $m\times n$ matrix-valued function 
such that $m\leq n$ and $t_{m-1}(G)>0$, then the Toeplitz operator $T_{zG}: H^{2}(\C^{n})\rightarrow 
H^{2}(\C^{m})$ has dense range.  A proof can be found in Chapter 14 of \cite{Pe1}.  This result was 
originally proved in the case of matrix-valued functions $G\in (H^{\infty}+C)(\M_{m,n})$ in \cite{PY1}.  
Recall that the Toeplitz operator $T_{\Psi}: H^{2}(\C^{m})\rightarrow H^{2}(\C^{n})$ with symbol 
$\Psi\in L^{\infty}(\M_{m,n})$ is defined by 
\[	T_{\Psi}f=\pP_{+}\Psi f,\;\mbox{ for }f\in H^{2}(\C^{n}),	\]
and $\pP_{+}$ denotes the orthogonal projection of $L^{2}(\C^{m})$ onto $H^{2}(\C^{m})$.

Let $k\geq0$, $\Phi\in L^{\infty}(\M_{m,n})$ and $\ell\geq 0$ be fixed.  It follows from Definition 
\ref{superoptDef} that if $Q\in\Omega_{\ell}^{(k)}(\Phi)$, then $\Zero$ belongs to $\Omega_{\ell}(\Phi-Q)$, 
$t_{j}(\Phi-Q)=t_{j}^{(k)}(\Phi)$, and $Q+F\in\Omega_{j}^{(k)}(\Phi)$ whenever $F\in\Omega_{j}(\Phi-Q)$ 
for $0\leq j\leq\ell$.  In particular, if $\Phi$ has a superoptimal approximant $Q$ in $\Hk(\M_{m,n})$, 
then $\Phi-Q$ is very badly approximable and $t_{j}(\Phi-Q)=t^{(k)}_{j}(\Phi)$ for all $j\geq0$.

In \cite{PT2}, Peller and Treil characterized admissible very badly approximable functions in terms 
of certain families of subspaces.  To state their result, let $\Psi$ be a matrix-valued function in 
$L^{\infty}(\M_{m,n})$ and $\sigma>0$.  For $\zeta\in\T$, we denote by $\fS_{\Psi}^{\sigma}(\zeta)$ 
the linear span of all Schmidt vectors of $\Psi(\zeta)$ that correspond to the singular values of 
$\Psi(\zeta)$ that are greater than or equal to $\sigma$.  Note that the subspaces 
$\fS_{\Psi}^{\sigma}(\zeta)$ are defined for almost all $\zeta\in\T$.
			
\begin{thm}[\cite{PT2}]\label{PTmain}
			Suppose $\Psi$ is an admissible matrix-valued function in $L^{\infty}(\M_{m,n})$.  Then $\Psi$ 
			is very badly approximable if and only if for each $\sigma>0$, there are functions $\xi_{1},\ldots, 
			\xi_{\ell}\in \ker T_{\Psi}$ such that
			\begin{equation*}
						\fS_{\Psi}^{\sigma}(\zeta)=\mySpan\{\,\xi_{j}(\zeta): 1\leq j\leq\ell\,\}\;\text{ for a.e. }
						\zeta\in\T.
			\end{equation*}
\end{thm}

For proofs of many of the previously mentioned results, we refer the reader to \cite{Pe1} and the 
references therein.

\section{Schmidt vectors of Hankel operators}\label{schmidtSection}

Henceforth, let $k$ be an integer such that $k\geq 1$ and $\Phi\in L^{\infty}(\M_{m,n})$.  In this section, 
we study the collection of Schmidt vectors 
\[	\Ek(\Phi)\Mydef\left\{\xi\in H^{2}(\C^{n}): H^{*}_{\Phi}H_{\Phi}\xi=s_{k}^{2}(H_{\Phi})\xi \right\}	\] 
which correspond to the singular value $s_{k}(H_{\Phi})$ of the Hankel operator $H_{\Phi}$.

To improve the transparency of some computations, we make use of the \emph{flip operator} 
$J:L^{2}(\C^{m})\rightarrow L^{2}(\C^{m})$, defined by	
\[	Jf=\bar{z}\bar{f}\;\mbox{ for }f\in L^{2}(\C^{m}).	\]  
It is easy to see that $J$ is an involution and satisfies
\[	JH_{\Phi}=H^{*}_{\Phi^{t}}J\;\text{ and }\; H_{\Phi}J=JH^{*}_{\Phi^{t}},	\]
because $J$ intertwines with the Riesz projections, i.e. $J\pP_{+}=\pP_{-}J$.  It follows that 
$s_{j}(H_{\Phi})=s_{j}(H_{\Phi^{t}})$ holds for all $j\geq0$ and 
\[	JH_{\Phi}\Ek(\Phi)= \Ek(\Phi^{t}).	\]

We make use of the following well-known lemma.
\begin{lemma}\label{spectralCons}
			Let $m\geq0$.  If $T\in\cB(\cH,\cK)$ satisfies $s_{m}(T)>\|T\|_{\rm e}$, then $s_{m}(T)$ is an 
			eigenvalue of $(T^{*}T)^{1/2}$.
\end{lemma}

Proof of this lemma can be based on the fact that any point $\lambda$ in the spectrum of $(T^{*}T)^{1/2}$ 
that does not belong to the essential spectrum of $(T^{*}T)^{1/2}$ must be an isolated eigenvalue of 
finite multiplicity of $(T^{*}T)^{1/2}$ (see Chapter XI in \cite{Co}).  This fact is a consequence of 
the Spectral Theorem for normal operators on a Hilbert space.

For the remainder of this section, we assume that
\begin{equation}\label{bestAssumption}
			\|H_{\Phi}\|_{\rm e}<s_{k}(H_{\Phi})<s_{k-1}(H_{\Phi}).
\end{equation}
Then Lemma \ref{spectralCons} implies that $s=s_{k}(H_{\Phi})$ is the $k$th largest eigenvalue of 
$(H_{\Phi}^{*}H_{\Phi})^{1/2}$ and has finite multiplicity $\mu=\dim \ker(H_{\Phi}^{*}H_{\Phi}-s^{2} I)$.  
Therefore,
\[	s_{k-1}(H_{\Phi})>s_{k}(H_{\Phi})=\ldots=s_{k+\mu-1}(H_{\Phi})>s_{k+\mu}(H_{\Phi}).	\]

Let $Q$ be a best approximant in $\Hk(\M_{m,n})$ to $\Phi$, $r\Mydef\rank H_{Q}\leq k$ and $B$ be a 
finite Blaschke-Potapov product such that $\ker H_{Q}=BH^{2}(\C^{n})$.  Clearly, $F\Mydef QB$ belongs 
to $H^{\infty}(\M_{m,n})$ and
\[	s_{k}(H_{\Phi})\leq \|H_{\Phi}|{B H^{2}(\C^{n})}\|=\|H_{\Phi B}\|\leq\|\Phi B- F\|_{L^{\infty}(\M_{m,n})}
										=\|\Phi-Q\|_{L^{\infty}(\M_{m,n})},	\]
because $B H^{2}(\C^{n})$ has co-dimension $r\leq k$ and $B$ takes unitary values on $\T$.  Formula 
$(\ref{skTreil})$ allows us to conclude that
\begin{equation}\label{skIdentity}
			s_{k}(H_{\Phi})=\|H_{\Phi B}\|=\|\Phi-Q\|_{L^{\infty}(\M_{m,n})}
\end{equation}
and so $\codim B H^{2}(\C^{n})=k$; otherwise, 
\[	s_{k-1}(H_{\Phi})\leq\|H_{\Phi}|{B H^{2}(\C^{n})}\|=s_{k}(H_{\Phi})	\] 
holds, contradicting the assumption $s_{k}(H_{\Phi})<s_{k-1}(H_{\Phi})$.

It is known that $\Ek(\Phi)$ consists of maximizing vectors of $H_{\Phi-Q}$ and $\Ek(\Phi)\subseteq\ker H_{Q}$ 
(Lemma 17.2 in Chapter 14 of \cite{Pe1}).  Therefore, if $\xi\in \Ek(\Phi)$ satisfies $\|\xi\|_{2}=1$, then
\[	s_{k}(H_{\Phi})=\|H_{\Phi}-H_{Q}\|=\|\pP_{-}(\Phi-Q)\xi\|_{2}\leq\|(\Phi-Q)\xi\|_{2}\leq\|\Phi-Q\|_{\infty}
		=s_{k}(H_{\Phi}).	\]
Thus, for \emph{any} $\xi\in\Ek(\Phi)$, it follows that
\begin{align}
			H_{\Phi}\xi=H_{\Phi-Q}\xi =&(\Phi-Q)\xi,\label{maxVector}\\
			\|(\Phi-Q)(\zeta)\,\xi(\zeta)\|_{\C^{n}}=&s_{k}(H_{\Phi})\|\xi(\zeta)\|_{\C^{n}}
						\;\text{ for a.e. }\zeta\in\T,\label{maxVecMatrix}\\
			\|H_{\Phi-Q}\|=\|\Phi-Q\|_{\infty}=&s_{k}(H_{\Phi})\;\text{ and }\;\|H_{\Phi-Q}\|_{\rm e}
						=\|H_{\Phi}\|_{\rm e}.\label{HGfacts}
\end{align}
The latter result in $(\ref{HGfacts})$ is an immediate consequence of the formula (Theorem 3.8 in Chapter 4 
of \cite{Pe1})
\[	\|H_{\Psi}\|_{\rm e}=\dist_{L^\infty(\M_{m,n})}(\Psi,(H^{\infty}+C)(\M_{m,n})),\;\text{ for }\Psi\in 
		L^{\infty}(\M_{m,n}).	\]

It is now easy to see that the following result holds.  

\begin{thm}\label{SchmidtVectors}
			Suppose $\Phi\in L^{\infty}(\M_{m,n})$ satisfies $(\ref{bestAssumption})$.  If $Q$ is a best 
			approximant in $H^{\infty}_{(k)}(\M_{m,n})$ to $\Phi$, then 
			\[	\Ek(\Phi)=\left\{\,\xi\in\ker H_{Q}:\, 	\|H_{\Phi-Q}\xi\|_{2}=s_{k}(H_{\Phi})\|\xi\|_{2},\, 
																							JH_{\Phi}\xi\in\ker H_{Q^{t}}	\,\right\}.	\]
\end{thm}
\begin{proof}
			Suppose $Q$ is any best approximant in $\Hk(\M_{m,n})$ to $\Phi$.  Recalling that transposition 
			does not change the norm of a matrix-valued function, we see that $Q^{t}$ is a best approximant 
			in $\Hk(\M_{m,n})$ to $\Phi^{t}$ by $(\ref{skTreil})$.  Therefore, $\Ek(\Phi)\subseteq\ker H_{Q}$
			and so
			\[	JH_{\Phi}\Ek(\Phi)=\Ek(\Phi^{t})\subseteq\ker H_{Q^{t}}.	\]
			Thus, $\xi\in\ker H_{Q}$, $JH_{\Phi}\xi\in\ker H_{Q^{t}}$, and
			$\|H_{\Phi-Q}\xi\|_{2}=s_{k}(H_{\Phi})\|\xi\|_{2}$ by $(\ref{maxVector})$ and $(\ref{HGfacts})$.
			
			On the other hand, if $\xi\in H^{2}(\C^{n})$ satisfies $\|H_{\Phi-Q}\xi\|_{2}=s_{k}(H_{\Phi})\|\xi\|_{2}$, 
			then $\xi$ is a maximizing vector of $H_{\Phi-Q}$ and so $H^{*}_{\Phi-Q}H_{\Phi-Q}\xi=s_{k}^{2}(H_{\Phi})\xi$, 
			because $\|H_{\Phi-Q}\|=s_{k}(H_{\Phi})$.  Therefore if, in addition, $\xi\in\ker H_{Q}$ and $H_{\Phi}\xi
			\in J\ker H_{Q^{t}}=\ker H^{*}_{Q}$, it must be that $H^{*}_{\Phi}H_{\Phi}\xi=s_{k}^{2}(H_{\Phi})\xi$ 
			and so $\xi\in\Ek(\Phi)$, as desired.
\end{proof}

\section{Fredholm Toeplitz operators and matrix weights}\label{FredholmSection}

In this section, we prove the following result.

\begin{thm}\label{FredholmThm}
			Let $\Psi\in L^{\infty}(\M_{n})$ be an admissible very badly approximable function.  If $\Psi$ has $n$ 
			non-zero superoptimal singular values of degree $0$, then the Toeplitz operator $T_{\Psi}$ is Fredholm 
			and
			\[	\ind T_{\Psi}=\dim \ker T_{\Psi}>0.	\]
\end{thm}
It is well-known and contained in the literature that admissible very badly approximable functions induce 
Toeplitz operators with non-trivial finite dimensional kernels (e.g. see Chapter 14 in \cite{Pe1}); however, 
we provide a simple proof of this fact using matrix-valued weights (defined below).  It is the hope 
of the author that this will clarify some of the (similar) ideas used in the next section.

Let $W$ be an $n\times n$ \emph{matrix-valued weight}; that is, a bounded matrix-valued function whose 
values are non-negative $n\times n$ matrices.  We define the weighted inner product
\[	(f,g)_{W}\Mydef\int_{\T}(W(\zeta)f(\zeta),g(\zeta))\dm(\zeta)\;\text{ for }f,g\in L^{2}(\C^{n}).	\]
As usual, $\|\cdot\|_{W}$ denotes the norm induced by the inner product $(\cdot,\cdot)_{W}$, i.e. 
$\|f\|_{W}^2=(f,f)_{W}$ for $f\in L^{2}(\C^{n})$.

Recall that an operator $T\in\cB(\cH,\cK)$ is said to be \emph{Fredholm} if $\Range T$ is closed in $\cK$, 
$\dim\ker T<\infty$, and $\dim\ker T^{*}<\infty$.  In this case, the \emph{index of $T$} is defined by
\[	\ind T\Mydef \dim\ker T-\dim\ker T^{*}.	\]
More information concerning Fredholm operators and index can be found in Chapter XI of \cite{Co}.

\begin{proof}[Proof of Theorem \ref{FredholmThm}]
			By Lemma \ref{spectralCons}, the admissibility of $\Psi$ guarantees that $\|H_{\Psi}\|^{2}$ is 
			an eigenvalue of $H_{\Psi}^{*}H_{\Psi}$ of finite multiplicity	and thus a maximizing vector of 
			$H_{\Psi}$. It follows that			
			\[	\ker T_{\Psi}=\{\,f\in H^{2}(\C^{n}): \|H_{\Psi} f\|_{2}=\|\Psi f\|_{2}\,\}	\]
			is non-empty; after all, since $\Psi$ is (very) badly approximable, then $\|H_{\Psi}\|=\|\Psi\|_{\infty}$ 
			and the inequalities
			\[	\|H_{\Psi}f\|_2\leq\|\Psi f\|_2\leq\|\Psi\|_{\infty}\|f\|_2=\|H_{\Psi}\|\cdot\|f\|_2	\]
			are all equalities for any maximizing vector $f$ of the Hankel operator $H_{\Psi}$.
			
			Let us show that $\dim\ker T_{\Psi}<\infty$.  To this end, consider the matrix-valued weight 
			$W=\Psi^{*}\Psi$.  Since $\Psi$ is an admissible very badly approximable function, it follows from
			$(\ref{sNumbersAndT})$ that $s_{j}(\Psi(\zeta))=t_{j}$ is non-zero (by the admissibility of $\Psi$) 
			for a.e. $\zeta\in\T$, $0\leq j\leq n-1$.  In particular, $W$ is invertible a.e. on $\T$ and
			$\|W(\zeta)^{-1}\|=s_{n-1}^{-1}(W(\zeta))=t_{n-1}^{-2}$ for a.e. $\zeta\in\T$.
			
			It is easy to see that
			\begin{align}
						\{\,\xi\in W^{1/2}H^{2}(\C^{n}): \|H_{\Psi}W^{-1/2}\xi\|_{2}=\|\xi\|_{2}\,\}&=
															\{\,f\in H^{2}(\C^{n}): \|H_{\Psi} f\|_{2}=\|W^{1/2} f\|_{2}\,\}\nonumber\\
													&=\{\,f\in H^{2}(\C^{n}): \|H_{\Psi} f\|_{2}=\|\Psi f\|_{2}\,\}\nonumber\\
													&=\ker T_{\Psi}\label{kernelForm}
			\end{align}
			and so the operator $H_{\Psi}W^{-1/2}$ defined on $W^{1/2}H^{2}(\C^{n})$ and equipped with the 
			$L^{2}$-norm (on its domain and range) has operator norm equal to 1.  Furthermore, the corresponding
			space of maximizing vectors of $H_{\Psi}W^{-1/2}$ equals $\ker T_{\Psi}$ because $\|H_{\Psi}f\|_2
			\leq\|\Psi f\|_2$ holds for all $f\in H^{2}(\C^n)$.  The essential norm of this operator also admits 
			the estimate
			\[	\|H_{\Psi}W^{-1/2}|W^{1/2}H^{2}(\C^{n})\|_{\rm e}\leq\|H_{\Psi}\|_{\rm e}\|W^{-1/2}\|_{\infty}
					=\|H_{\Psi}\|_{\rm e}t_{n-1}^{-1}(\Psi)<1	\]
			due to the admissibility of $\Psi$.  Therefore, the space of maximizing vectors of the operator 
			$H_{\Psi}W^{-1/2}|W^{1/2}H^{2}(\C^{n})$	is finite dimensional.  In view of $(\ref{kernelForm})$, 
			we deduce now that $\ker T_{\Psi}$ is non-empty and finite dimensional.

			By Theorem 5.4 in Chapter 14 of \cite{Pe1}, the Toeplitz operator $T_{z\Psi}^{*}$ has trivial kernel 
			and so $\ker T_{\Psi}^{*}$ is trivial as well.  Therefore it suffices to show that $T_{\Psi}$ has 
			closed range.  To this end, let
			\begin{equation}\label{tauDef}
						\tau_{\Psi}\Mydef\sup\left\{ \|H_{\Psi}f\|_{2}: \|f\|_{W}=1, f\in(\ker T_{\Psi})^{\bot}\right\}.
			\end{equation}
			Clearly, $\tau_{\Psi}\leq 1$ because  $\|H_{\Psi}f\|_2\leq\|f\|_W$ for all $f\in H^{2}(\C^n)$.  
			Moreover, the trivial identity $\|\Psi f\|_{2}^{2}=\|T_{\Psi}f\|_{2}^2+\|H_{\Psi}f\|_{2}^2$ is valid 
			for all $f\in H^{2}(\C^{n})$ and implies that
			\[	\|f\|_{W}^{2}\leq\|T_{\Psi}f\|_{2}^2+\tau_{\Psi}^2\|f\|_{W}^2
					\;\text{ for }f\in (\ker T_{\Psi})^{\bot},	\]
			or equivalently,
			\[	(1-\tau_{\Psi}^2)\|f\|_{W}^2	\leq\|T_{\Psi}f\|_{2}^2\;\text{ for }f\in (\ker T_{\Psi})^{\bot}.	\]
			Since the matrix-valued weight $W$ satisfies the inequality 
			\[	\|f\|_{W}\geq t_{n-1}\|f\|_{2}\;\text{ for }f\in H^{2}(\C^{n}),	\]
			then
			\begin{equation}
						(1-\tau_{\Psi}^2)t_{n-1}^{2}\|f\|_{2}^{2}	\leq\|T_{\Psi}f\|_{2}^2
						\;\text{ for }f\in (\ker T_{\Psi})^{\bot}.
			\end{equation}
			Thus, the Toeplitz operator $T_{\Psi}$ is bounded from below on $(\ker T_{\Psi})^{\bot}$ \emph{if}
			$\tau_{\Psi}<1$.  In particular, this implies that the Toeplitz operator $T_{\Psi}$ has closed range
			because	the restriction $T_{\Psi}$ to $(\ker T_{\Psi})^{\bot}$ does.  To complete the proof, it 
			remains to show that $\tau_{\Psi}<1$.  The following lemma is needed, the proof of which can be deduced 
			from the ideas used above and Lemma \ref{spectralCons}.
			
			\begin{lemma}[\cite{PT2}]\label{PT2Lemma}
						Let $W$ be an invertible admissible weight for a Hankel operator $H_{\Psi}$ such that $W(\zeta)
						\geq a^2 I$, $a>\|H_{\Psi}\|_{\rm e}$, and let $K$ be a closed subspace of $H^{2}(\C^n)$.  If 
						\[	q=\sup\{\|H_{\Psi}f\|: f\in K, \|f\|_{W}= 1 \}	\]
						equals 1, then there exists a (non-zero) vector $f_{0}\in K$ such that $\|H_{\Psi}f_{0}\|_{2}
						=\|f_0\|_W$.
			\end{lemma} 
			
			By Lemma \ref{PT2Lemma}, if $K=(\ker T_{\Psi})^{\bot}=\{ f\in H^{2}(\C^{n}): \|H_{\Psi}f\|_2
			=\|f\|_{W}\}^{\bot}$, $W=\Psi^*\Psi$ (as before), $a=t_{n-1}$, and $\tau_{\Psi}=1$, then there 
			is an $f_{0}\in K$ such that $\|H_{\Psi}f_0\|_2=\|f_0\|_{W}$ and so $f_{0}\in\ker T_{\Psi}$, a 
			contradiction to our choice of $K$.  This completes the proof of Theorem \ref{FredholmThm}.
\end{proof}

The following corollary is a well-known consequence of results concerning ``thematic'' factorizations of 
very badly approximable functions (see \cite{PY1} and \cite{PY3}, or Chapter 14 in \cite{Pe1}).  Additionally,
it is now a consequence of Theorem \ref{FredholmThm}.
  
\begin{cor}\label{FredholmThmCor}
			If $\Psi$ satisfies the hypotheses of Theorem \ref{FredholmThm}, then $\dim \ker T_{\Psi}\geq n$.	
\end{cor}
\begin{proof}
			By Theorem \ref{FredholmThm}, $T_{\Psi}$ is Fredholm and so $T_{z\Psi}$ is Fredholm as well.  
			Moreover, the assumptions on $\Psi$ imply that the Toeplitz operator $T_{z\Psi}$ has dense 
			range (see section \ref{verybadSection}) and thus $\Range T_{z\Psi}=H^{2}(\C^{n})$.  It follows 
			that, for each $c\in\C^{n}$, there is an $f\in H^{2}(\C^{n})$ such that $T_{z\Psi}f=c$ and so 
			$T_{\Psi}f=\pP_{+}\bar{z}\pP_{+}(z\Psi)f=0$, i.e. $f\in\ker T_{\Psi}$.  Hence $\dim\ker T_{\Psi}\geq n$.
\end{proof}

\begin{cor}\label{preIndex}
			Suppose $\Phi$ is $k$-admissible and $Q$ is the superoptimal approximant in $\Hk(\M_{n})$ of $\Phi$.  
			If the number of non-zero superoptimal singular values of degree $k$ of $\Phi$ equals $n$, then the 
			Toeplitz operator $T_{\Phi-Q}$ is Fredholm and $\ind T_{\Phi-Q}=\dim\ker T_{\Phi-Q}>0$.
\end{cor}
\begin{proof}
			As observed in section \ref{verybadSection}, the matrix-valued function $\Psi=\Phi-Q$ is very 
			badly approximable and $t_{j}(\Psi)=t_{j}^{(k)}(\Phi)$ for all $j\geq 0$.  In particular, 
			$\Psi$ is admissible and has $n$ non-zero superoptimal singular values.  Hence, the conclusion 
			follows from Theorem \ref{FredholmThm}.
\end{proof}

\section{$k$-Admissible weights for Hankel operators}\label{weightSection}

\begin{definition}
			Given a Hankel operator $H_{\Psi}: H^{2}(\C^{n})\rightarrow H^{2}_{-}(\C^{m})$ and a matrix-valued 
			weight $W$,	we say that $W$ is a \emph{$k$-admissible weight for $H_{\Psi}$} if the inequality
			\[	\|H_{\Psi} f\|_{2}\leq\|f\|_{W}	\]
			holds for all $f$ which belong to an invariant subspace of $H^{2}(\C^{n})$ under multiplication by 
			$z$ whose codimension is at most $k$.  
\end{definition}
Equivalently, $W$ is a $k$-admissible weight if the operator $T_{W}-H_{\Psi}^{*}H_{\Psi}$ is non-negative
on an invariant subspace of $H^{2}(\C^{n})$ under multiplication by $z$ whose codimension is at most $k$.  
(Note that our definition of $k$-admissibility for weights is stated differently than in \cite{T2} yet it 
is equivalent.)

The reason for studying $k$-admissible weights for Hankel operators arises naturally from the problem of 
best approximation in $\Hk(\M_{m,n})$.  Indeed for $\Phi\in L^{\infty}(\M_{m,n})$, it is not difficult to 
see (by Nehari's theorem, i.e. equation $(\ref{skTreil})$ with $k=0$) that finding a best approximant $Q$ 
in $\Hk(\M_{m,n})$ to $\Phi$ is equivalent to finding a non-trivial (closed) invariant subspace $\cM$ of 
$H^{2}(\C^{n})$ under multiplication by $z$ with codimension at most $k$ such that $\|H_{\Phi}f\|_{2}\leq 
s_{k}(H_{\Phi})\|f\|_{2}$ for all $f\in\cM$.  That is, the best approximation problem in $\Hk(\M_{m,n})$ 
is equivalent to verifying that the matrix-valued weight $W=s_{k}^{2}(H_{\Phi})I$ is $k$-admissible for 
the Hankel operator $H_{\Phi}$, where $I$ denotes the identity on $H^{2}(\C^{n})$.

The collection of $k$-admissible weights for Hankel operators was characterized by Treil.

\begin{fact}[\cite{T2}]\label{AAKwithW}
			Suppose $\Phi\in L^{\infty}(\M_{m,n})$.  A matrix-valued weight $W$ is $k$-admissible for $H_\Phi$ 
			if and only if there is a $Q\in \Hk(\M_{m,n})$ such that \[	(\Phi-Q)^{*}(\Phi-Q)\leq W.	\]
\end{fact}

The proof is contained in that of Theorem 16.1 in \cite{T2}.  Actually, one does not need to use spectral 
measures nor the Iokhvidov-Ky Fan theorem; the version above can be deduced from the case (of the Weighted 
Nehari Problem) $k=0$ (see Theorem 6.1 in \cite{Pe1}) due to the characterization of the invariant subspaces 
of multiplication by $z$ on $H^{2}(\C^{n})$ with finite codimension stated in section \ref{bestRemarks}.  

Given $\Phi\in L^{\infty}(\M_{m,n})$, an $n\times n$ matrix-valued weight $W$, and a subspace $\cL$ of 
$H^{2}(\C^{n})$, we define
\begin{equation}\label{EwDef}
			\cE_{W}(\Phi;\cL)\Mydef\{ \xi\in\cL: \|H_{\Phi}\xi\|_{2}=\|\xi\|_{W}\}.
\end{equation}

A few observations are in order.  Let $\Phi\in L^{\infty}(\M_{m,n})$ satisfy $(\ref{bestAssumption})$ and 
$Q$ be a best approximant in $\Hk(M_{m,n})$ to $\Phi$.  If $W\Mydef(\Phi-Q)^{*}(\Phi-Q)$,  then
\begin{enumerate}
			\item	the operator $T_{W}-H^{*}_{\Phi}H_{\Phi}$ is non-negative when restricted to $\ker H_{Q}$, i.e. 
						$W$ is a $k$-admissible weight for $H_{\Phi}$,
			\item	the kernel of $(T_{W}-H^{*}_{\Phi}H_{\Phi})|\ker H_{Q}$ coincides with $\cE_{W}(\Phi;\ker H_{Q})$,
			\item	the space of maximizing vectors of the operator 
						\begin{equation}\label{HankelWithWeight}
									H_{\Phi}:(\ker H_{Q},\|\cdot\|_{W})\rightarrow (H^{2}_{-}(\C^{n}),\|\cdot\|_{2})
						\end{equation}
						equals $\cE_{W}(\Phi;\ker H_{Q})$, and
			\item	$H_{\Phi}\xi=(\Phi-Q)\xi$ and $H_{\Phi}\xi\in J\ker H_{Q^{t}}$ hold for $\xi\in
						\cE_{W}(\Phi;\ker H_{Q})$.
\end{enumerate}
\begin{remark}
			It is easy to see that $\cE_{W}(\Phi;\ker H_{Q})$ contains the non-empty finite dimensional subspace
			$\Ek(\Phi)$ (e.g. see equation $(\ref{maxVector})$) however these sets need not be equal in general.  
			On the other hand, if $W$ equals the identity a.e. on $\T$, then $\Ek(\Phi)=\cE_{W}(\Phi;\ker H_{Q})$ 
			holds (by Lemma \ref{schmidtEW} below).  
\end{remark}

We now provide a suitable extension of Theorem \ref{PTmain} to the class of $k$-admissible functions.

\begin{thm}\label{CNTheorem}
			Suppose $\Phi$ is $k$-admissible and $Q$ is the superoptimal approximant in $\Hk(\M_{n})$ of $\Phi$.  
			If the number of non-zero superoptimal singular values of degree $k$ of $\Phi$ equals $n$, then there 
			are functions $\xi_{1},\ldots, \xi_{\ell}\in\cE_{W}(\Phi;\ker H_{Q})$, where $W=(\Phi-Q)^{*}(\Phi-Q)$, 
			such that
			\begin{equation}\label{CNconclusion}
						\fS_{\Phi-Q}^{\sigma}(\zeta)=\mySpan\{\,\xi_{j}(\zeta): 1\leq j\le \ell\,\}\;\text{ for a.e. }
						\zeta\in\T.
			\end{equation}
\end{thm}

\begin{remark}\label{impRemark}
			The conclusion of Theorem \ref{CNTheorem} is \emph{not} an immediate consequence of Theorem \ref{PTmain}. 
			Indeed, the very bad approximability of $\Phi-Q$ \emph{only} implies that there are finitely many 
			functions $\xi_{1},\ldots, \xi_{\ell}\in \ker T_{\Phi-Q}$ such that $(\ref{CNconclusion})$ holds; 
			however, the conclusion of Theorem \ref{CNTheorem} states that one can choose these functions in 
			$\cE_{W}(\Phi;\ker H_{Q})$.  In fact, $\cE_{W}(\Phi;\ker H_{Q})=\ker H_{Q}\cap \ker T_{\Phi-Q}$ 
			is a subset of $\ker T_{\Phi-Q}$, but these sets are not equal when $k>0$.
\end{remark}

The proof of Theorem \ref{CNTheorem} follows from ideas in the ``matrix-weight proof'' of Theorem 4.1 in 
\cite{PT2}.  Some of these arguments were used earlier in \cite{T2} in a different context.

\begin{proof}[Proof of Theorem \ref{CNTheorem}]
			Let $\cL\Mydef\ker H_{Q}$.  The matrix-valued weight $W$ is a $k$-admissible weight for the Hankel 
			operator $H_{\Phi}$ on $\cL$, as observed above.  Moreover, in view of the choice of the superoptimal 
			approximant $Q$ of $\Phi$, $\Phi-Q$ is very badly approximable and $(\ref{sNumbersAndT})$ holds.

			Let $\sigma_{0},\ldots,\sigma_{r}$ denote all of the \emph{distinct} superoptimal singular values 
			of degree $k$ of $\Phi$ arranged in decreasing order.
			
			For $0\leq j\leq r$, define the functions $\Lambda_{j}(x)=\max\{x,\sigma_{j}^{2}\}$ for $x\geq 0$
			and the weights $W_{j}(\zeta)=\Lambda_{j}((\Phi-Q)^{*}(\zeta)(\Phi-Q)(\zeta))$ for $\zeta\in\T$.
			
			Notice that $W(\zeta)\leq W_{j}(\zeta)$ for a.e. $\zeta\in\T$ and so the weight $W_{j}$ is 
			$k$-admissible for the Hankel operator $H_{\Phi}$.  Let $\cE_{j}\Mydef\cE_{W_{j}}(\Phi;\cL)$ for 
			$0\leq j\leq r$.  
			Then $\Ek(\Phi)\subseteq\cE_{0}$ and $\cE_{0}\subset\ldots\subset\cE_{r}=\cE_{W}(\Phi;\cL)$.  Moreover, 
			for each $0\leq j\leq r$, $W_{j}(\zeta)$ is invertible a.e. on $\T$ and $\|W_{j}(\zeta)^{-1/2}\|=
			\sigma_{j}$ for a.e. $\zeta\in\T$ because
			\begin{equation}\label{normWInv}
						\|W_{j}(\zeta)^{-1}\|=s_{n-1}^{-1}(W_{j}(\zeta))=\sigma_{j}^{-2}\;\text{ for a.e. }\zeta\in\T.
			\end{equation}
			
			Fix $0\leq j\leq r$.  By considering the operator $H_{\Phi}W_{j}^{-1/2}$ on $W_{j}^{1/2}\cL$ (equipped 
			with the $L^{2}$-norm), we see that 
			\[	\|H_{\Phi}W_{j}^{-1/2}|{W_{j}^{1/2}\cL}\|=\|H_{\Phi}:(\cL,\|\cdot\|_{W_{j}})
					\rightarrow (H^{2}_{-}(\C^{n}),\|\cdot\|_{2})\|=1	\]
			and
			\[	\|H_{\Phi}W_{j}^{-1/2}|{W_{j}^{1/2}\cL}\|_{\rm e}\leq\|H_{\Phi}\|_{\rm e}\|W_{j}^{-1/2}\|_{\infty}
					\leq\|H_{\Phi}\|_{\rm e}\sigma_{j}^{-1}<1.	\]
			Therefore the space of maximizing vectors $\cE_{j}$, as defined in $(\ref{EwDef})$, of the operator
			$H_{\Phi}W_{j}^{-1/2}|{W_{j}^{1/2}\cL}$ is non-empty and finite dimensional.  Let
			\[	q_{j}\Mydef\|H_{\Phi}|\cL\ominus_{W_{j}}\cE_{j}\|=\sup\{\,\|H_{\Phi}\xi\|_{2}: \xi\in\cL\ominus_{W_{j}}
					\cE_{j}, \|\xi\|_{W_{j}}=1\,\},	\] 
			where
			\[	\cL\ominus_{W_{j}}\cE_{j}\Mydef\{\, g\in\cL: (f,g)_{W_{j}}=0\text{ for all }f\in\cE_{j}\,\}. 	\] 
			Clearly, $q_j\leq 1$.  If $q_{j}=1$, then the restriction to $\cL\ominus_{W_{j}}\cE_{j}$ of the 
			operator in $(\ref{HankelWithWeight})$ must also have a maximizing vector by Lemma \ref{PT2Lemma},			
			a contradiction as $\cE_{j}$ contains all maximizing vectors.  Hence, it must be that $q_{j}<1$.
						
			For $\zeta\in\T$, let $E_{j}(\zeta)\Mydef\mySpan\{\,f(\zeta): f\in\cE_{j}\,\}$.  It is easy to verify 
			that $\dim E_{j}(\zeta)$ equals a constant for a.e. $\zeta\in\T$ (e.g. see Chapter VII in \cite{He}). 
			Therefore, it suffices to show that $E_{j}(\zeta)=\fS_{\Phi-Q}^{\sigma_{j}}(\zeta)$ for a.e. $\zeta\in\T$.
						
			Assume that there is a function $f\in\cE_{j}$ such that $f(\zeta)\notin\fS_{\Phi-Q}^{\sigma_{j}}(\zeta)$ 						
			on a set of positive measure.  Since $\dim \fS_{\Phi-Q}^{\sigma_{j}}(\zeta)$ equals a constant a.e. on 
			$\T$, it follows that $f(\zeta)\notin\fS_{\Phi-Q}^{\sigma_{j}}(\zeta)$ holds for a.e. $\zeta\in\T$.  
			However, this implies that $\|(\Phi-Q)(\zeta)f(\zeta)\|_{\C^{n}}<\sigma_{j}\|f(\zeta)\|_{\C^{n}}$ and so
			\[	\|H_{\Phi}f\|_{2}\leq\|(\Phi-Q)f\|_{2}<\sigma_{j}\|f\|_{2}\leq\|f\|_{W_{j}},	\]
			because $f\in\cL$, a contradiction to the assumption $f\in\cE_{j}$.  Thus, $E_{j}(\zeta)
			\subseteq\fS_{\Phi-Q}^{\sigma_{j}}(\zeta)$ for a.e. $\zeta\in\T$.  
			
			Assume now for the sake of contradiction that the subspace-valued function $E_{j}(\zeta)$ is a proper 
			subspace of $\fS_{\Phi-Q}^{\sigma_{j}}(\zeta)$ on a set of positive measure.  In this case, it must be 
			that $E_{j}(\zeta)$ is a proper subspace of $\fS_{\Phi-Q}^{\sigma_{j}}(\zeta)$ for \emph{a.e.}
			$\zeta\in\T$ because $\dim \fS_{\Phi-Q}^{\sigma_{j}}(\zeta)$ and $\dim E_{j}(\zeta)$ are constant 
			for a.e. $\zeta\in\T$.
			
			Consider now the one-parameter family of weights $W^{[a]}$, $a>0$, defined by
			\[	W^{[a]}(\zeta)\Mydef P_{E_{j}(\zeta)}W_{j} P_{E_{j}(\zeta)}+a^{2}P_{E_{j}(\zeta)^{\bot}},	\]
			where $P_{E_{j}(\zeta)}$ and $P_{E_{j}(\zeta)^{\bot}}$ denote the orthogonal projections from $\C^{n}$ 
			onto $E_{j}(\zeta)$ and $E_{j}(\zeta)^{\bot}$, respectively, for a.e. $\zeta\in\T$.
					
			If $f\in\cE_{W_{j}}$ and $g\in\cL\ominus_{W_{j}}\cE_{W_{j}}$, then $H_{\Phi}f\bot H_{\Phi} g$ and so
			\[	\|H_{\Phi}(f+g)\|_{2}^{2}=\|H_{\Phi}f\|_{2}^{2}+\|H_{\Phi}g\|_{2}^{2}=\|f\|_{W_{j}}^{2}
																	+\|H_{\Phi}g\|_{2}^{2}\leq \|f\|_{W_{j}}^{2}+q_{j}^2\|g\|_{W_{j}}^{2}.	\]
			Let $a=\sigma_{j} q_{j}>0$.  It is not difficult to show that
			\[	\|f\|_{W_{j}}=\|f\|_{W^{[a]}}\;\text{ and }\;q_{j}\|g\|_{W_{j}}\leq\|g\|_{W^{[a]}}	\]
			(c.f. the ``matrix-weight proof'' of Theorem 4.1 in \cite{PT2}).  Therefore, 
			\begin{equation}\label{WaAdmissible}
						\|H_{\Phi}(f+g)\|_{2}^{2}\leq \|f\|_{W^{[a]}}^{2}+\|g\|_{W^{[a]}}^{2}=\|f+g\|_{W^{[a]}}^{2}
			\end{equation}
			as the orthogonal complements of $\cE_{W_{j}}$ in $\cL$ with respect to the inner products
			$(\cdot,\cdot)_{W_{j}}$ and $(\cdot,\cdot)_{W^{[a]}}$ are equal.
			 
			The inequality in $(\ref{WaAdmissible})$ implies that the weight $W^{[a]}$ is $k$-admissible for 
			the Hankel operator $H_{\Phi}$ on $\cL$ and so, by Theorem \ref{AAKwithW}, there is a matrix-valued 
			function $Q_{\#}\in\Hk(\M_{n})$ such that 
			\[	(\Phi-Q_{\#})^{*}(\Phi-Q_{\#})\leq W^{[a]}.	\]
			
			Let $N$ be the largest integer such that $s_{N}((\Phi-Q)(\zeta))=\sigma_{j}$ for a.e. $\zeta\in\T$.  
			Then, for any $j<N$,
			\[	s_{j}((\Phi-Q_{\#})(\zeta))\leq s_{j}(W^{[a]}(\zeta))^{1/2}
					\leq s_{j}(W(\zeta))^{1/2}=s_{j}((\Phi-Q)(\zeta))	\]
			and 
			\[	s_{N-1}((\Phi-Q_{\#})(\zeta))\leq s_{N-1}(W^{[a]}(\zeta))^{1/2}=a<\sigma_{j}=s_{N-1}((\Phi-Q)(\zeta))	\]
			hold for a.e. $\zeta\in\T$, contradicting the choice of the \emph{superoptimal} approximant $Q$.  Hence
			$E_{j}(\zeta)=\fS_{\Phi-Q}^{\sigma_{j}}(\zeta)$ for a.e. $\zeta\in\T$, as desired.
\end{proof}

\begin{cor}\label{schmidtEW}
			Let $\Phi$ be $k$-admissible, $Q$ be the superoptimal approximant in $\Hk(\M_{n})$ of $\Phi$, and 
			$W=(\Phi-Q)^{*}(\Phi-Q)$.  If all of the superoptimal singular values of degree $k$ of $\Phi$ are 
			equal, then 
			\begin{equation}\label{EWvsEk}
						\cE_{W}(\Phi;\ker H_{Q})=\Ek(\Phi),
			\end{equation}
			and $E_{*}(\zeta)\Mydef\Myspan \{\, f(\zeta):\, f\in\Ek(\Phi)\, \}$ equals $\C^{n}$ for a.e. $\zeta\in\T$.
\end{cor}
\begin{proof}
			By assumption, every superoptimal singular value of degree $k$ of $\Phi$ equals $s_{k}(H_{\Phi})$ 
			and so $W(\zeta)=s_{k}(H_{\Phi})I$ a.e. on $\T$ (see $(\ref{sNumbersAndT})$), where $I$ denotes the 
			$n\times n$ identity matrix valued function.  Thus if $\xi\in\cE_{W}(\Phi;\ker H_{Q})$, then $\xi$ 
			is a maximizing vector of $H_{\Phi-Q}$ and so the equality in $(\ref{EWvsEk})$ holds by Theorem 
			\ref{SchmidtVectors}.
			
			Let $\sigma=s_{k}(H_{\Phi})$.  The statement concerning $E_{*}(\zeta)$ is now trivial in view of Theorem 
			\ref{CNTheorem} and $(\ref{EWvsEk})$; after all, $E_{*}(\zeta)=\fS_{\Phi-Q}^{\sigma}(\zeta)=\C^{n}$ 
			for a.e. $\zeta\in\T$.
\end{proof}

\section{The Index formula}\label{indexSection}

Let $\Phi$ be $k$-admissible and $Q$ be the superoptimal approximant in $\Hk(\M_{n})$ of $\Phi$.  Recall 
that from the remarks following Lemma \ref{spectralCons}, there are finite Blaschke-Potapov products $B$ 
and $\Lambda$, of degree $k$, such that 
\[	\ker H_{Q}=B H^{2}(\C^{n})\;\text{ and }\;\ker H_{Q^{t}}=\Lambda H^{2}(\C^{n}).	\]
In addition, the weight $W\Mydef(\Phi-Q)^{*}(\Phi-Q)$ is a $k$-admissible weight for $H_{\Phi}$ because 
the operator $T_{W}-H_{\Phi}^{*}H_{\Phi}$ is non-negative on $\ker H_{Q}$ and $\codim\ker H_{Q}=k$.

A key observation is made in the following theorem.

\begin{thm}\label{mainObserThm}
			Suppose $\Phi$ is $k$-admissible and $Q$ is the superoptimal approximant in $\Hk(\M_{n})$ of $\Phi$.  
			If the number of non-zero superoptimal singular values of degree $k$ of $\Phi$ equals $n$, then the
			matrix-valued function $U\Mydef\Lambda^{t}(\Phi-Q)B$ is very badly approximable and 
			\begin{equation}\label{mainObser}
						\cE_{W}(\Phi;\ker H_{Q})=B\ker T_{U}.
			\end{equation}
			In particular, $\dim \cE_{W}(\Phi;\ker H_{Q})\geq n$.
\end{thm}
\begin{proof}
			We first establish the equality in $(\ref{mainObser})$ holds.  Recall that if 
			$\xi\in\cE_{W}(\Phi;\ker H_{Q})$, then $(\Phi-Q)\xi=H_{\Phi}\xi\in J \ker H_{Q^{t}}$ and $\xi\in\ker H_{Q}$.  
			It follows that there is an $f\in H^{2}(\C^{n})$ such that $\xi=B f$ and
			\[	U f=\Lambda^{t}(\Phi-Q)B f=\Lambda^{t}H_{\Phi}\xi\in \Lambda^{t}J\ker H_{Q^{t}}=H^{2}_{-}(\C^{n});	\]
			that is, $\cE_{W}(\Phi;\ker H_{Q})\subseteq B\ker T_{U}$.  On the other hand, it is evident that $\ker T_{U}
			\subseteq\ker T_{(\Phi-Q)B}$ and 
			\begin{align*}
						\cE_{W}(\Phi;\ker H_{Q})&=\{\, Bf: f\in H^{2}(\C^{n})\text{ and }\|H_{\Phi}Bf\|_{2}=\|Bf\|_{W}\,\}\\
																		&=B\{\, f\in H^{2}(\C^{n}): \|H_{(\Phi-Q)B}f\|_{2}=\|(\Phi-Q)Bf\|_{2}\,\}\\
																		&=B\ker T_{(\Phi-Q)B}.
			\end{align*}
			Therefore, $B\ker T_{U}\subseteq\cE_{W}(\Phi;\ker H_{Q})$ and thus the proof of $(\ref{mainObser})$ is 
			complete.
				
			Let $\sigma>0$.  In view of the equalities $s_{j}((\Phi-Q)(\zeta))=s_{j}(U(\zeta))$, $j\geq 0$, it is 
			easy to see that $B\fS_{U}^{\sigma}(\zeta)=\fS_{\Phi-Q}^{\sigma}(\zeta)$ for a.e. $\zeta\in\T$.  
			
			By Theorem \ref{CNTheorem}, there are $\xi_{1},\ldots,\xi_{\ell}\in\cE_{W}(\Phi;\ker H_{Q})=B\ker T_{U}$
			such that
			\[	B\fS_{U}^{\sigma}(\zeta)=\fS_{\Phi-Q}^{\sigma}(\zeta)=\mySpan\{\,\xi_{j}(\zeta): 1\leq j\le \ell\,\}
					\text{ for a.e. }\zeta\in\T.	\]
			We deduce now from Theorem \ref{PTmain} that the matrix-valued function $U$ is very badly approximable.
			By Corollary \ref{FredholmThmCor}, we conclude $\dim\cE_{W}(\Phi;\ker H_{Q})\geq n$.
\end{proof}

We now prove our main result.

\begin{thm}\label{MyMainThm}
			Suppose $\Phi$ is $k$-admissible and $Q$ is the superoptimal approximant in $\Hk(\M_{n})$ of $\Phi$.  
			If the number of non-zero	superoptimal singular values of degree $k$ of $\Phi$ equals $n$, then the 
			Toeplitz operator $T_{\Phi-Q}$ is Fredholm and 
			\begin{equation}\label{indexFor}
						\ind T_{\Phi-Q}=\dim\ker T_{\Phi-Q}=2k+\dim\cE_{W}(\Phi;\ker H_{Q}).
			\end{equation}
\end{thm}
\begin{proof}
			By Corollary \ref{preIndex}, the Toeplitz operator $T_{\Phi-Q}$ is Fredholm and the first equality in 
			$(\ref{indexFor})$ holds.  Therefore, the matrix-valued function $U=\Lambda^{t}(\Phi-Q)B$ induces a 
			Fredholm Toeplitz operator $T_{U}$ with index
			\begin{equation}\label{indU1}
						\ind T_{U}=\ind T_{\Phi-Q}-2k.
			\end{equation}
			By virtue of Theorem \ref{mainObserThm}, $U$ is very badly approximable and so
			\begin{equation}\label{indU2}
						\ind T_{U}=\dim\ker T_{U}=\dim\cE_{W}(\Phi;\ker H_{Q}).
			\end{equation}
			The conclusion follows immediately from $(\ref{indU1})$ and $(\ref{indU2})$.
\end{proof}

Finally, we state two simple consequences of Theorem \ref{MyMainThm}.

\begin{cor}
			Suppose $\Phi$ is an $n\times n$ matrix-valued function, continuous on $\T$, such that $s_{k}(H_{\Phi})
			<s_{k-1}(H_{\Phi})$.  Let $Q$ be the superoptimal approximant in $\Hk(\M_{n})$ of $\Phi$.  If $\det(\Phi-Q)
			\neq 0$ a.e. on $\T$, then 
			\[	\dim\ker T_{\Phi-Q}\geq 2k+n.	\]
\end{cor}
\begin{proof}
			In view of the equation in $(\ref{sNumbersAndT})$, the assumption $\det(\Phi-Q)\neq 0$ a.e. on $\T$ 
			implies $\Phi$ has $n$ non-zero superoptimal singular values of degree $k$.  The conclusion now follows 
			from Theorem \ref{MyMainThm}.
\end{proof}

\begin{cor}\label{indexCor}
			Suppose $\Phi$ is $k$-admissible and all of the superoptimal singular values of degree $k$ of $\Phi$ 
			are equal.  If $Q$ is a best approximant in $\Hk(\M_{n})$ to $\Phi$, then 
			\begin{equation*}
						\dim\ker T_{\Phi-Q}=2k+\mu,
			\end{equation*}
			where $\mu$ denotes the multiplicity of the singular value $s_{k}(H_{\Phi})$.
\end{cor} 
\begin{proof}
			The formula is an immediate consequence of Theorem \ref{MyMainThm} and Lemma \ref{schmidtEW}.
\end{proof}

\textbf{Acknowledgment.}  The author would like to thank Professor S.R. Treil for useful comments 
concerning superoptimal approximation by meromorphic functions.

\end{document}